\documentclass[a4paper, 12pt]{article}
\usepackage{authblk}
\usepackage{calc, amsmath, amsthm, a4, latexsym, amssymb, color, url}
\usepackage{graphicx}
\usepackage{mathrsfs}
\usepackage[usenames,dvipsnames,svgnames,table]{xcolor}
\usepackage[english]{babel} 
\usepackage{babelbib}
\usepackage{tikz}
\definecolor{comcolor}{rgb}{0.9,0.3,0.3}
\definecolor{starcolor}{rgb}{0.3,0.3,0.9}
\definecolor{hscolor}{rgb}{0.9,0.6,0.5}
\definecolor{darkgreen}{rgb}{0.1,0.6,0.3}

\newtheorem{thm}{Theorem}[section]
\newtheorem{lemma}[thm]{Lemma}
\newtheorem{corollary}[thm]{Corollary}
\newtheorem{prop}[thm]{Proposition}

\theoremstyle{definition}

\newtheorem{rem}[thm]{Remark}

\newcommand{\be}[1]{\begin{equation}\label{#1}}
\newcommand{\ee}{\end{equation}}
\newcommand{\ba}{\begin{array}}
\newcommand{\ea}{\end{array}}
\newcommand{\bal}{\begin{aligned}}
\newcommand{\eal}{\end{aligned}}

\newcommand{\E}{\mathbb{E}}

\renewcommand{\P}{\mathbb{P}}

\newcommand{\Var}{\mathrm{Var}}

\title{
\textbf{The shape of a seed bank tree}}
\author{Adri\'an Gonz\'alez Casanova$^a$, Lizbeth Pe\~naloza$^b$, Arno Siri-J\'egousse$^b$\\
\small{$^a$ Instituto de Matem\'aticas, $^{b}$IIMAS, Departamento de Probabilidad y Estad\'istica.}\\
\small{Universidad Nacional Aut\'onoma de M\'exico.}}
\begin{document}
\maketitle
\begin{abstract}
We derive the asymptotic behavior of the total, active and inactive branch lengths of the seed bank coalescent, when the size of the initial sample grows to infinity.
Those random variables have important applications for populations evolving under some seed bank effects, such as plants and bacteria, and for some cases of structured populations { like} metapopulations.
The proof relies on the study of the tree at a stopping time corresponding to the first time that a deactivated lineage reactivates. 
We also give conditional sampling formulas for the random partition and we study the system at the time of the first deactivation of a lineage.
All these results provide a good picture of the different regimes and behaviors of the block-counting process of the seed bank coalescent.
\end{abstract}

\small{{\bf Keywords}: Seed bank, Structured coalescent, Branch lengths, Sampling formula.}

\section{Introduction}

Seeds, cysts and other forms of dormancy generate seed banks, which store genetic information that can be temporally lost from a population at a certain time and \textit{resuscitate} later. Having a seed bank is a prevalent evolutionary strategy which has important consequences. For example, in the case of bacteria, it buffers against the selective pressure caused by environmental variability and at the same time increases genetic variation \cite{GC2014, LJ,T2011}. 

A first attempt to construct a probabilistic model to study this phenomenon is due to Kaj, Krone and Lascoux \cite{K2001}.
They considered a modified Wright-Fisher model in which each individual \textit{chooses} its parent from the individuals at several generations in the past, and not only from the previous one. This construction has an important technical complication arising from the loss of the Markov property. 
A new model was defined and studied in \cite{Blath2016} to avoid this issue.
It consists in a two-level discrete Markov chain, which again generalizes the Wright-Fisher model.

Consider a haploid population of fixed size $N$ which supports a seed bank of constant size $M$. The $N$ \textit{active} individuals are called \textit{plants} and the $M$ \textit{dormant} individuals are called \textit{seeds}.
Let $0\leq\varepsilon\leq1$ be such that $\lfloor\varepsilon N\rfloor\leq M$.
The $N$ plants from generation 0 produce $N$ individuals in generation $1$ by multinomial sampling (as in the Wright-Fisher model). However, $N-\lfloor\varepsilon N\rfloor$ randomly chosen of these individuals are plants and $\lfloor\varepsilon N\rfloor$ are seeds.
Then $\lfloor\varepsilon N\rfloor$ uniformly (without replacement) sampled seeds from the seed bank in generation 0 become plants in generation 1. The $\lfloor\varepsilon N\rfloor$ seeds produced by the plants in generation 1, take the place of the seeds that germinate. Thus, we have again $N$ plants and $M$ seeds in generation 1 (see Figure \ref{sbm}).
This random mechanism is repeated independently to produce the next generations.
Observe that this model has, unlike \cite{K2001}, non-overlapping reproduction events.

\begin{figure}
	\centering
		\includegraphics[scale=0.3]{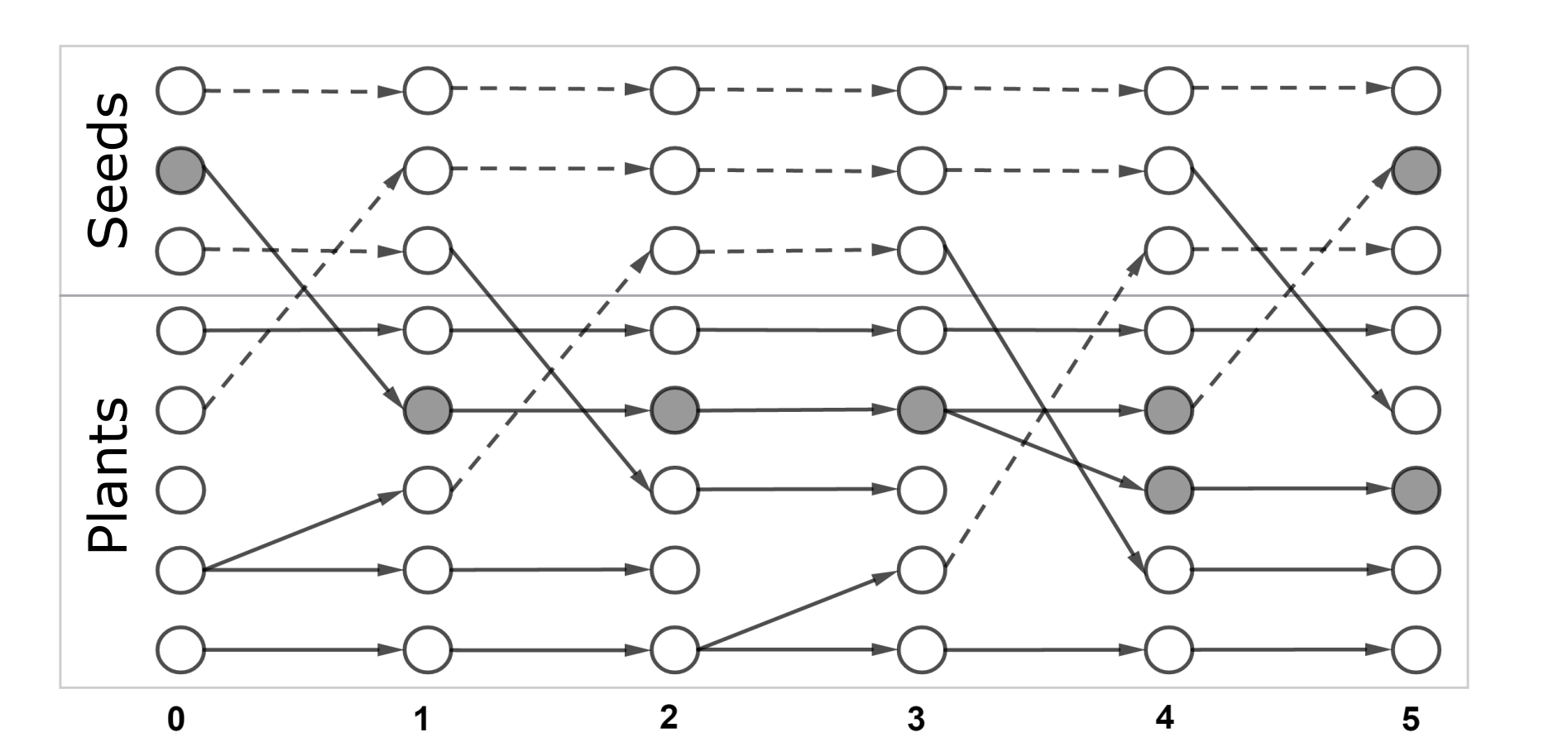}
	\caption{The discrete seed bank model. In this picture $N=5$, $M=3$ and $\lfloor\varepsilon N\rfloor$=1, i.e., in each generation four plants are produced by active individuals, one seed germinates and one new seed is produced.}
	\label{sbm}
\end{figure}

If we let $N$ (and $M$) go to infinity and rescale the time, the stochastic process that describes the limiting gene genealogy of a sample taken from the seed bank model is called the \textit{seed bank coalescent} \cite{Blath2016}.
Apart from populations of plants or bacteria, it is remarkable that the seed bank coalescent is a convenient genealogical model for some metapopulations \cite{L2013}. In fact, it was independently introduced in that context and named the \textit{peripatric coalescent}. It corresponds to a special modification of structured coalescence in which small colonies can emerge from a main population and merge again with it.
The seed bank coalescent is a structured coalescent with an active part, having the dynamics of a \textit{Kingman coalescent}, and a dormant part where the lineages are like frozen. Lineages can activate or deactivate at certain rates, see Figure \ref{figsc} for an illustration.
\begin{figure}
	\centering
		\includegraphics[scale=0.2]{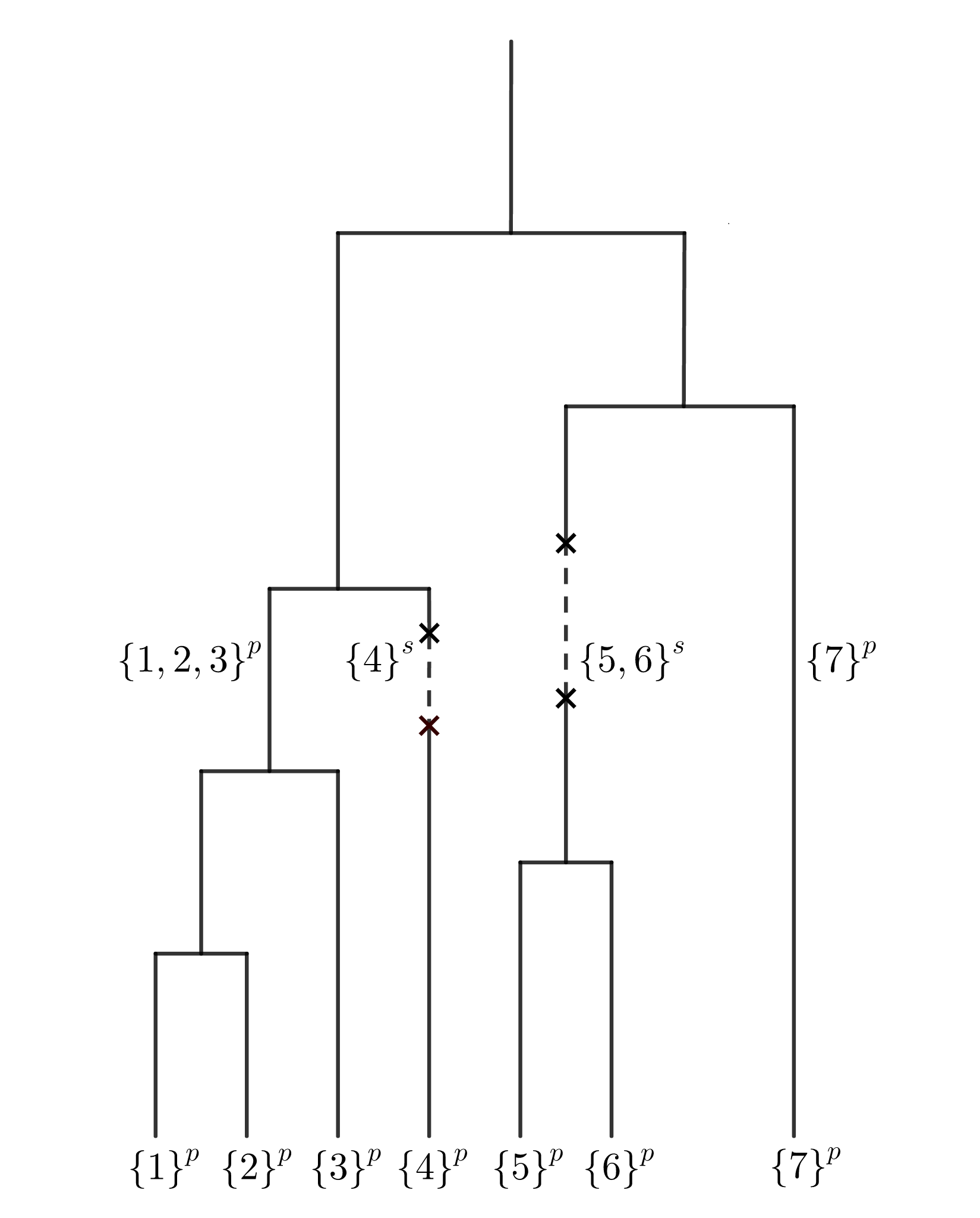}
	\caption{A possible realization of the seed-bank 7-coalescent. Dotted lines indicate inactive individuals and the crosses mean that there is a deactivation or a reactivation. }
	\label{figsc}
\end{figure}

In this paper we study the asymptotic behavior of some functionals of the seed bank tree.
 These can be useful for genetic applications, but also they provide a light shed on the connections between theory and applications. As an illustration, there is a close relation between the shape of the genealogical tree of a sample of size $n$ and the number of mutations observed in it. More precisely, suppose that mutations appear in the genealogy by simply superimposing a Poisson process on the ancestral lineages (as it is in the infinite sites model, see Chapter 1.4 in \cite{Durrett2008}). Then, the shape of the tree determines the distribution of the data obtained by DNA sequencing and thus, it can be inferred from it.
For example, conditionally on the total length of the coalescent, denoted by $L_n$, the number of mutations observed in the sample has Poisson distribution with parameter $\mu L_n$, where $\mu$ is the mutation rate. Thus, if we know the asymptotic behavior of the total length of the tree we can deduce the asymptotic behavior of the number of mutations. This is the key tool for obtaining a Watterson-type estimator for the mutation rate, see \cite{Durrett2008}. 
Not surprisingly, asymptotics of the total length of many classical coalescents have been studied, e.g. in \cite{DIMR,DDSJ, Ker12, DK}.

 In \cite{Blath2016}, it was established that the time to the most recent common ancestor of a sample of size $n$ in the seed bank coalescent is of order $\log\log n$. This is an important difference with the classical {Kingman coalescent}, whose height is finite.
 In our study, we establish that the total length of the tree built from a sample of $n$ plants and zero seeds is of the same order than that of the Kingman coalescent, behaving like $\log n$, but with a different multiplicative constant depending on the activation and deactivation parameters of the model.
 Moreover, we show that the total active length behaves precisely like the total length of the Kingman coalescent. 
 This means that it is technically very hard to distinguish between the null Kingman model and the alternative seed bank model, unless the dormant individuals have the possibility of mutate while being in the seed bank, that is actually the case in the metapopulation framework described in \cite{L2013}. To discriminate both null and seed bank models, some finer results such as sampling formulas can also be derived.
  We are able to describe the seed back tree in detail as it undergoes different phases. Indeed, it can be said that we describe the shape of the seed bank tree.
 
Our results also have practical implications. In \cite{Maughan}, Maughan observed experimentally that a population of bacteria undergoing dormancy typically does not have significantly different number of mutations. Our findings agree with this observation and offer new insights on the reason for this: most of the mutations occur in the {\it Kingman phase} i.e. shortly before the leaves of the tree, and in this part of the ancestral tree dormancy is irrelevant. On the other hand, populations suffering a significant amount of mutations while being in the dormant state would be expected to have a higher evolutionary rate. This remark together with \cite{Maughan} suggests that the mutations that occur to individuals in latent state play a minor role (at least number-wise). This is  opposed to previous works suggesting that the normal rate of molecular evolution of bacteria with a seed bank is evidence that mutations affecting dormant individuals are frequent \cite{Maughan}.

 \subsection{Main results}
 We study some relevant stopping times of the seed bank coalescent, leading to a complete description of the shape of the tree and explaining how long the genealogies  spend in successive dynamical phases, as is detailed precisely in Table \ref{summarytable} and Figure \ref{summaryfig}.

Let us now define properly the seed bank coalescent. Fix $n\in\mathbb N$ and let $\mathcal{P}_n$ be the set of  partitions of $[n]:=\lbrace1,2,...,n\rbrace$. Then, the set of marked partitions $\mathcal{P}_n^{\lbrace p,s \rbrace}$ is built out from $\mathcal{P}_n$ by adding a flag (either $p$ for a plant or $s$ for a seed) to each block of the partition. For example, for $n=7$, $\pi=\lbrace\lbrace 1,2 ,3\rbrace^{p},\lbrace4\rbrace^{s},\lbrace 5,6\rbrace^{s},\lbrace 7 \rbrace^{p}\rbrace$ is an element of  $\mathcal{P}_7^{\lbrace p,s \rbrace}$.
The seed bank $n$-coalescent $(\Pi_{n}(t))_{t\geq0}$, with deactivation intensity $c_1>0$ and activation intensity $c_2>0$, is the continuous-time Markov chain with values in $\mathcal{P}_n^{\lbrace p,s \rbrace}$ having the following dynamics.
As for the Kingman coalescent, each pair of plant blocks merges at rate 1, independent of each other. 
Moreover, any block can change its flag, from $p$ to $s$ at rate $c_1$, and vice versa at rate $c_2$, see Figure \ref{figsc} for an illustration.

The block-counting process of the seed bank $n$-coalescent is the two-dimensional Markov chain $(N_n(t),M_n(t))_{t\geq0}$ with values in $([n]\cup\{0\})\times([n]\cup\{0\})$ and the following transition rates, for $t\geq0$.
\begin{equation*}\label{seed}
(N_n(t),M_n(t))\text{ jumps from }(i,j)\text{ to}\left\{ \begin{array}{lll}
(i-1,j), & \textrm{ at rate } \binom{i}{2} & \mbox{(coalescence)},\\
(i-1,j+1), & \textrm{ at rate } c_1i &\mbox{(deactivation)},\\
(i+1,j-1), & \textrm{ at rate } c_2j &\mbox{(activation)}.
\end{array}
\right .
\end{equation*}

Note that $N_n(t)$  can have either an upward jump if a seed becomes a plant, or a downward jump if there is a coalescence event or a plant becomes a seed.
Also observe that each jump has size one.
In the sequel, we suppose that $N_n(0)=n$ and $M_n(0)=0$.

 For $i\in[n]$, we denote by $\tau_n^i$ the reaching time of the level $i$ by the process $N_n$, i.e. ${\tau}_n^n=0$ and
 \begin{equation}\label{deftau}
 \tau_n^i=\inf\{t\geq0:N_n(t)=i\}.
 \end{equation} 
 Furthermore, let $\gamma_n$ and $\theta_n$ be, respectively, the first time that some plant becomes a seed and the first time that some seed becomes a plant, i.e.
  \begin{equation}\label{defgamma}
  \gamma_n=\mbox{inf}\lbrace t>0: M_n({t-})<M_n(t)\rbrace=\mbox{inf}\lbrace t>0: M_n(t)=1\rbrace
  \end{equation}
   and 
   \begin{equation}\label{deftheta}
   \theta_n=\inf\{t>0:M_n(t-)>M_n(t)\}.
   \end{equation}
 Finally, denote by $\sigma_n$ the time to the most recent common ancestor, already studied in \cite{Blath2016},
 $$\sigma_n=\mbox{inf}\lbrace t>0: N_n(t)+M_n(t)=1\rbrace=\mbox{inf}\lbrace t>0: N_n(t)=1,M_n(t)=0\rbrace.$$

We first obtain asymptotic results on the random variables $\gamma_n$ and $\theta_n$ and the size of the system
at those times. The results obtained in Section \ref{secgamma} and \ref{sectheta} can be summarized in Table \ref{summarytable} and Figure \ref{summaryfig}.

\begin{table}%\label{summarytable}
\begin{center}
\begin{tabular}{lllll}
\cline{1-4}
\multicolumn{1}{|l|}{  {\footnotesize Stopping time ($\tau$)}} & \multicolumn{1}{l|}{{\footnotesize Asymptotics of $\tau$}} & \multicolumn{1}{l|}{{\footnotesize Asymptotics of $N_n(\tau)$}} & \multicolumn{1}{l|}{{\footnotesize Asymptotics of $M_n(\tau)$}} &  \\ \cline{1-4}
\multicolumn{1}{|l|}{$\gamma_n$}              & \multicolumn{1}{l|}{$2(1-Y)/Yn$}         & \multicolumn{1}{l|}{$Yn$}         & \multicolumn{1}{l|}{1}         &  \\ \cline{1-4}
\multicolumn{1}{|l|}{$\theta_n$}              & \multicolumn{1}{l|}{$T/\log n$}         & \multicolumn{1}{l|}{$Z\log n$}         & \multicolumn{1}{l|}{$2c_1\log n$}         &  \\ \cline{1-4}
   \multicolumn{1}{|l|}{$\sigma_n$}                                 &        \multicolumn{1}{l|}{$\log\log n$}                       &         \multicolumn{1}{l|}{$1$}                      &          \multicolumn{1}{l|}{$0$}                     & \\ \cline{1-4}
                                    &                               &                               &                               &
\end{tabular}
\end{center}
\caption{Summary of the asymptotic behavior of the functional of the seed bank coalescent studied in this work. Here $Y$ is a $Beta(2c_1,1)$ distributed random variable, $T$ is an exponential random variable with parameter $2c_1c_2$ and $Z$ is a Fr\'echet random variables with shape parameter 1 and scale parameter $4c_1c_2$.}
\label{summarytable}
\end{table}
\begin{figure}[ht!]
	\centering
		\includegraphics[scale=.4]{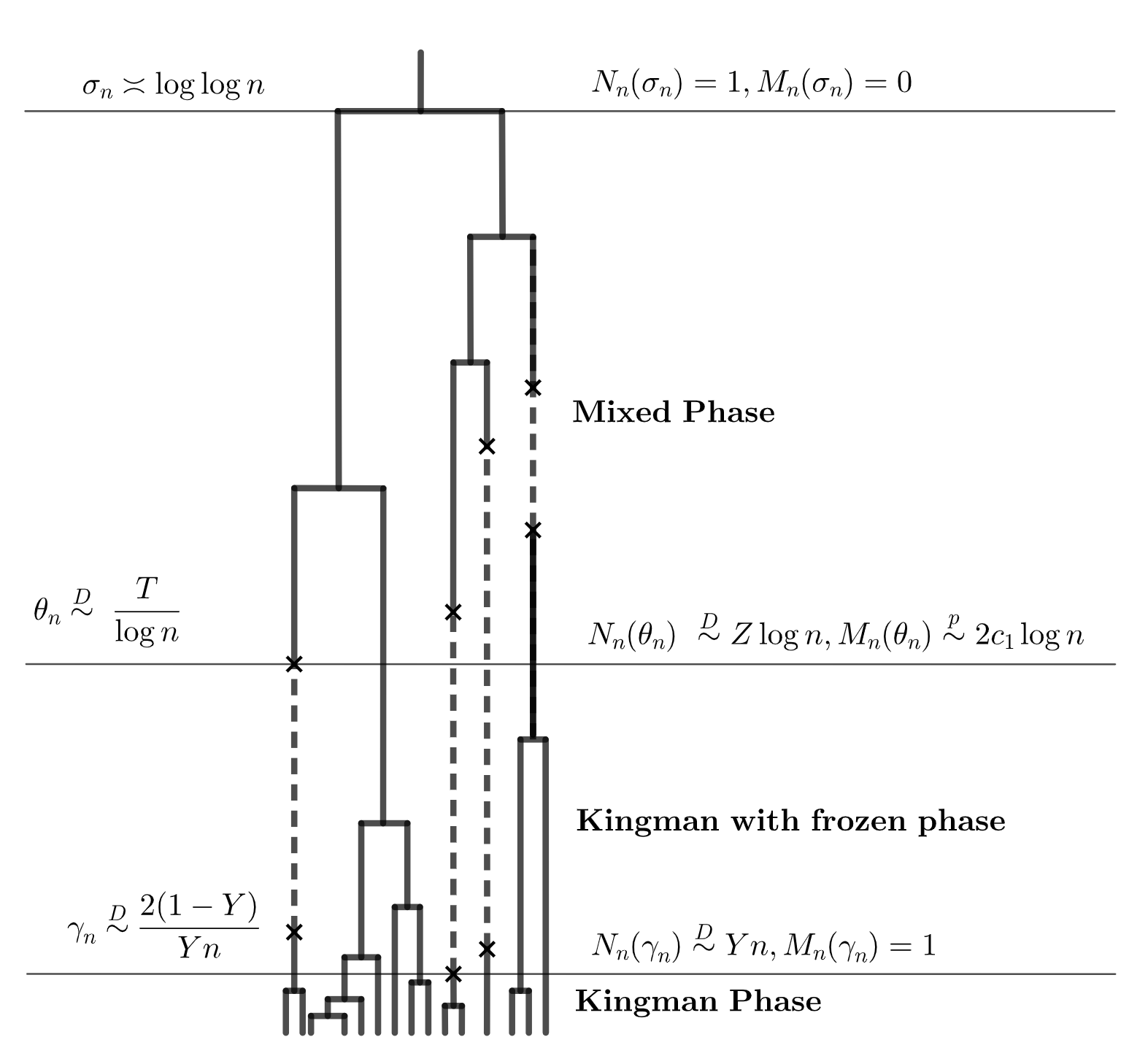}
	\caption{Summary of the asymptotic behavior of the functionals of the seed bank coalescent studied in this work. Here $Y$ is a $Beta(2c_1,1)$ distributed random variable,$T$ is an exponential random variable with parameter $2c_1c_2$ and $Z$ is a Fr\'echet random variables with shape parameter 1 and scale parameter $4c_1c_2$. The symbol $A_n\overset{p}\sim B_n$ means that $\frac{A_n}{B_n}\to1$ in probability. The symbol $A_n\overset{\mathcal D}\sim XB_n$ means that $\frac{A_n}{B_n}\to X$ in distribution. %The symbol $A_n\overset{p}\asymp B_n$ means that $\P(B_n^{1+\varepsilon}\leq{A_n}\leq{B_n}^{1-\varepsilon})\to 1$ for every $\varepsilon$.
	The symbol $A_n\asymp B_n$ means that $C_1B_n\leq\E[{A_n}]\leq C_2{B_n}$ for some constants $C_1,C_2$.
	}
	\label{summaryfig}
\end{figure}
In Section 4 we analyze the total length 
\begin{equation}\label{defL}
L_n=A_n+I_n
\end{equation}
where the {\it active length} is defined by
\begin{equation}\label{defA}
A_n=\int_0^{\sigma_n}N_n(t)dt
\end{equation}
and the {\it inactive length} by
\begin{equation}\label{defI}
I_n=\int_0^{\sigma_n}M_n(t)dt.
\end{equation}
Our main result is stated as follows.
\begin{thm}\label{thmL}
Consider the seed bank coalescent starting with $n$ plants and no seeds. Then,
\begin{equation*}
\lim_{n\rightarrow\infty}\frac{L_n}{\log n}=2\left(1+\frac{c_1}{c_2}\right)
\end{equation*} 
in probability.
\end{thm}
Interestingly, numerical techniques of \cite{HSJB} used to study the total length for fixed $n$ show that the balance between active and inactive lengths is equally conserved for their expectations for any $n\geq2$,
$${c_1}\E[A_n]=c_2\E[I_n].$$
The behavior of both $A_n$ and $I_n$ is obtained by considering those variables before and after the time of the first activation $\theta_n$.
Hence, results of Section \ref{sectheta} are key tools for the forthcoming proofs. 
Theorem \ref{thmL} also gives an immediate corollary on the number of active and inactive mutations on the seed bank tree.

\begin{corollary} Consider the seed bank coalescent starting with $n$ plants and no seeds. Let $S_n$ be the number of mutations in the seed bank tree and let $\mu$ be the mutation rate. Then 
\begin{equation*}
\lim_{n\rightarrow\infty}\frac{S_n}{\log n}=2\mu\left(1+\frac{c_1}{c_2}\right)
\end{equation*} 
in probability.
\end{corollary}

%\begin{proof}
%We know that $S_n$ has Poisson distribution with parameter $\mu L_n$
%\end{proof}

 Finally, in Section 5, we establish a sampling formula which is inspired by Watterson's ideas in \cite{W1983} and which help us to understand the fine configuration of the blocks of a seed bank coalescent at given times.

 \section{The time of the first deactivation}\label{secgamma}

We start with the study of $\gamma_n$, the time of the first deactivation defined in \eqref{defgamma}, and the size of the system at this time.
%we consider the number of coalescence events until $\gamma_n$.
% More precisely, recall $(\tau_n^i)_{i=1}^n$ in \eqref{deftau} and define $$T_n=\mbox{inf}\lbrace i\geq1:\tau_{n}^{i}=\gamma_n\rbrace.$$ 
 Observe that, if $N_n(0)=n$ and $M_n(0)=0$, there are $n-N_n(\gamma_n){-1}$ coalescence events until time $\gamma_n$
 and we can write
 \[\gamma_n=\sum_{i=N_n(\gamma_n)+1}^{n}V_i\] where the $V_i$'s are independent exponential random variables with respective parameters $\binom{i}{2}+c_1i$.

We start with an easy limit result on the variable $N_n(\gamma_n)$.
Note that, by considering $c_1$ as a mutation rate, $n-N_n(\gamma_n){-1}$ can also be interpreted as the number of coalescence events before the most recent mutation in the genealogy. Recent studies on the shape of coalescent trees at the time of the first mutation in a branch can be found in \cite{FS20P}, with some direct applications to coalescent model selection \cite{FS20S}.  

\begin{prop}\label{propT} Consider a seed bank coalescent starting with $n$ plants and no seeds. Then, 
$$\lim_{n\rightarrow\infty}\frac{N_n(\gamma_n)}{n}= Y$$
in distribution, where $Y\sim Beta(2c_1,1)$.
\end{prop}

\begin{proof} Let $z\in(0,1)$. We have that 

\begin{eqnarray*}
\P(N_n(\gamma_n)\leq zn)&=&\prod_{i=\lfloor zn\rfloor+1}^{n}\frac{\binom{i}{2}}{\binom{i}{2}+c_1i}
=\prod_{i=\lfloor zn\rfloor}^{n-1}\dfrac{i}{i+2c_1}\\
&=&\exp \left\lbrace  -\sum_{i=\lfloor zn\rfloor}^{n-1}\log\left( 1+\dfrac{2c_1}{i}\right)\right\rbrace.
\end{eqnarray*}
Using that $\log(1+x)\sim x$ near 0, we obtain
\begin{eqnarray*}
\P(N_n(\gamma_n)\leq zn)&\sim&\exp \left\lbrace  -\sum_{i=\lfloor zn\rfloor}^{n-1}\frac{2c_1}{i}\right\rbrace\\
&\sim&\exp \left\lbrace{-2c_1}\log\left(\frac{1}{z}\right)\right\rbrace\\
&=&z^{2c_1}
\end{eqnarray*}
which is the distribution function of a $Beta(2c_1,1)$ random variable.
\end{proof}

Now, let us establish the asymptotic behavior of the time of the first deactivation, $\gamma_n$.
\begin{prop}\label{teorgama} 
Consider a seed bank coalescent starting with $n$ plants and no seeds. Then,
\begin{equation}\label{cvgamma}
\lim_{n\rightarrow\infty}n\gamma_n=\Gamma:=\dfrac{2(1-Y)}{Y}
\end{equation}
in distribution, 
where $Y$ is $Beta(2c_1,1)$ distributed. 
The density function of $\Gamma$ is \[f_\Gamma(x)=c_1\left(\dfrac{2}{2+x}\right)^{2c_1+1}\]
for $x\geq0$.
In particular, if $c_1>{1}/{2}$, then the expectation of $\Gamma$ is finite
\[\mathbb{E}[\Gamma]=\frac{2}{2c_1-1} \] 
and if $c_1>1$, the variance of $\Gamma$ is finite
\[\Var(\Gamma)=\frac{4c_1}{(c_1-1)(2c_1-1)^2}.\] 
\end{prop}
\begin{proof}
 Let $G_n(0)=0$ and, for $t\in (0,1)$, define
$$
G_n({t})=\sum_{i=\lfloor (1-t)n \rfloor+1}^nV_i=\sum_{i=\lfloor (1-t)n \rfloor+1}^n\dfrac{2e_i}{i(i-1+2c_1)},
$$
where the $e_i$'s are i.i.d standard exponential random variables. 
With this notation, we obtain $\gamma_n=G_n\left(1-N_n(\gamma_n)/n\right).$

We first show that, for any $t\in (0,1)$, we have
\begin{equation}\label{dos}
\lim_{n\rightarrow\infty}\left( nG_n(s)\right)_{s\leq t}= \left( \dfrac{2s}{1-s}\right)_{s\leq t}
\end{equation}
in distribution, in the sense of weak convergence in the path space D[0,t].
To this, let us first establish that, for a fixed $t\in(0,1)$,
\begin{equation}\label{cvAL2}
\lim_{n\rightarrow\infty}nG_{n}(t)= \dfrac{2t}{1-t}
\end{equation}
in $L^2$.
By definition, we have that 
\begin{eqnarray*}
\mathbb{E}[nG_n(t)]
&=&\sum_{i=\lfloor (1-t)n \rfloor+1}^{n}\frac{2n}{i(i-1+2c_1)}\\
&\sim &  \frac{1}{n}\sum_{i=\lfloor (1-t)n \rfloor+1}^{n}\frac{2}{(i/n)^2}.
\end{eqnarray*}
By a Riemann sum argument, we obtain that
$$\mathbb{E}[nG_n(t)]\sim\int_{1-t}^1  \dfrac{2}{x^{2}}dx=\dfrac{2t}{1-t}.
$$
Now, by the independence of the random variables $e_i$,
\begin{eqnarray*}
\Var(nG_n(t))
&=&\sum_{i=\lfloor (1-t)n \rfloor+1}^{n}\frac{4n^2}{i^2(i-1+2c_1)^2}\\
&\sim&\sum_{i=\lfloor (1-t)n \rfloor+1}^{n}\frac{4n^2}{i^4}.
\end{eqnarray*}
Again, by a Riemann sum argument, we obtain that
$\Var(nG_n(t))$ converges to 0 as $n\to\infty$.
This gives \eqref{cvAL2}.

To obtain \eqref{dos} we follow the same steps as those of Proposition 6.1 in \cite{Dhersin2013}, with $\alpha=2$. 
Then, the proof of \eqref{cvgamma} follows by adapting the alternative proof of Theorem 5,2 in \cite{Dhersin2013}, p. 1713, taking $\alpha=2$ and the limit variable $\sigma$ being $1-Y$ and $Beta(1, 2c_1)$ distributed.

The distribution function of $\Gamma$ is given by
\begin{eqnarray*}
\mathbb{P}\left( \Gamma\leq x\right)&=& \mathbb{P}\left( Y\geq \frac{2}{2+x}\right)\\ \\
&=&1-\left(\frac{2}{2+x} \right)^{2c_1}
\end{eqnarray*}
for $x\geq0$.
We get the density by differentiating.
The moments of $\Gamma$ are obtained by computing
$$\E[\Gamma^k]=\int_0^\infty kx^{k-1}\P(\Gamma>x)dx=\int_0^\infty kx^{k-1}\left(\frac{2}{2+x} \right)^{2c_1}dx.$$
In particular, the $k$th moment is finite for $c_1>k/2$.
\end{proof}

\section{The time of the first activation}\label{sectheta}

In this section we study $\theta_n$, the first time that a seed becomes a plant, which we introduced in \eqref{deftheta}.
We also provide some limit laws for $N_n({\theta_n})$ and $M_n({\theta_n})$.
Observe that from time zero up to time $\theta_n$ only two types of events occur, either coalescence or deactivation. 
Recall the successive reaching times of the chain $N_n$, denoted by $(\tau_{n}^i)_{i=1}^n$ and defined in \eqref{deftau}. 
\begin{prop}\label{summarythetan}
Consider a seed bank coalescent starting with $n$ plants and no seeds.
Then, the following asymptotics hold.
\begin{equation}\label{limitN}
 \lim_{n\rightarrow\infty} \frac{N_{n}({\theta_n})}{\log n}=Z
\end{equation}
in distribution, where $Z$ is a Fr\'echet random variable with shape parameter 1 and scale parameter $4c_1c_2$, with distribution function $\P(Z\leq z)=\exp\{-4c_1c_2/z\}$.
Also
\begin{equation}\label{limitM}
 \lim_{n\rightarrow\infty} \frac{M_{n}({\theta_n})}{\log n}=2c_1
\end{equation}
in probability.
Finally,
\begin{equation}\label{limittheta}
\lim_{n\rightarrow\infty} {\log n\theta_n}=T
% \lim_{n\rightarrow\infty}\mathbb{P}\left( {(\log n)^{-1-\varepsilon}}<{\theta_n}<{(\log n)^{-1+\varepsilon}}\right)=1.
\end{equation}
in distribution, where $T$ is an exponential random variable with parameter $2c_1c_2$.
\end{prop}

The proof of \eqref{limitM} is obtained by combining Lemmas \ref{cotasupM} and \ref{cotainfM}. 
The proof of \eqref{limitN} and \eqref{limittheta} is obtained by combining Lemmas \ref{theta} and \ref{cotainfN} which appear in the sequel.
We get these results by coupling the seed bank coalescent with two simpler models.

The {\it coloured} seed bank coalescent (see Definition 4.2 in \cite{Blath2016}) is a marked coalescent where additionally each element of $[n]$ has a flag indicating its color: white or blue. 
%Recall $\mathcal{P}_n^{\{s,p\}}$ to be the set of marked partitions of $[n]$ where the seed bank coalescent takes its values, then the colored coalescent will take values in $\mathcal{P}_n^{\{s,p\}\times\{w,b\}}$. 
%As an example, if an element $\pi$ of $\mathcal{P}_8^{\{s,p\}}$ is the marked partition $$\pi=\{\{1,2\}^p,\{3\}^s,\{4,5,6\}^p,\{7,8\}^s\},$$ then a colored partition of $\mathcal{P}_8^{\{s,p\}\times\{w,b\}}$ can be $${\pi}'=\{\{1^w,2^b\}^p,\{3^w\}^s,\{4^w,5^b,6^b\}^p,\{7^w,8^b\}^s\}.$$
 Movements and mergers of the blocks of the colored coalescent follow the same dynamics as those of the classical seed bank coalescent. Additionally, if a block activates, each individual inside this block gets the color blue. In other cases colors remain unchanged.
%Following the example, if $\{7^w,8^b\}^s$ is reactivated, the resulting block is $\{7^b,8^b\}^p$.

As in \cite{Blath2016}, we start with all individuals colored with white, so color blue only appears after a reactivation event, and we also use the notation $\underbar N_{n}(t)$ (resp. $\underbar M_{n}(t)$) for the number of white plants (resp. white seeds) at time $t$, starting with $n$ (white) plants and zero seeds. 

The notation for the reaching times of $\underbar{N}_{n}$ are $\underline{\tau}_n^n=0$ and, for $i\in[n-1]$,
$$\underline{\tau}_n^i=\inf\lbrace t>0:\underbar{N}_n({t})=i\rbrace.$$

Note that, on the event $\{\tau_n^i<\theta_n\}$, we have $\underline \tau_n^i=\tau_n^i$ a.s., and in general the stochastic bound
\begin{equation}\label{tau>bartau}
\underline{\tau}_n^{i-1}-\underline{\tau}_n^{i}\leq_{st}\tau_n^{i-1}-\tau_n^{i} 
\end{equation}
holds.

This model is of particular use to prove that the number of seeds that ``survive"  up to moment $\theta_n$ is of order $\log n$.
More precisely, as in \cite{Blath2016}, consider the independent Bernoulli random variables $B^i_{n}=$ $\textbf{1}_{\lbrace\mbox{deactivation at }\underline{\tau}_n^i\rbrace}$, for $i\in[n-1]$, with respective parameter
\begin{eqnarray}\label{beeni}
\mathbb{P}(B^i_{n}=1)&=&\frac{c_1(i+1)}{\binom{i+1}{2}+c_1(i+1)}\nonumber\\
&=&\dfrac{2c_1}{i+2c_1},
\end{eqnarray}
independently of the number of seeds in the system. It is clear that, almost surely for any $t\geq0$, $M_n(t)\leq\sum_{i=1}^{n-1}B_n^i$. This and Bienaym\'e-Chebyshev's inequality lead to the following straightforward result.

\begin{lemma}\label{cotasupM}
For any $\varepsilon>0$,
\begin{equation}\label{supM}
 \mathbb{P}\left(\sup_{t\geq0} M_{n}({t})>2c_1(1+\varepsilon)\log n\right)\leq\frac{1}{2c_1\varepsilon^2\log n}.
\end{equation}

In particular, for any $\varepsilon>0$,
\begin{equation*}
 \lim_{n\rightarrow\infty}\mathbb{P}\left( M_{n}({\theta_n})\leq2c_1(1+\varepsilon)\log n\right)=1.
\end{equation*}

\end{lemma}

The {\it bounded} seed bank coalescent is a modification of the original seed bank coalescent, where only $m$ seeds can be accumulated in the bank.
Thus, when the bank is full, a deactivating lineage disappears instead of moving to the bank. In our case, we start with $n$ plants and $m$ seeds (the bank is full from the beginning). 

Denote by $\bar{N}_{n,m}(t)$ (resp. $\bar{M}_{n,m}(t)$) for the number of plants (resp. seeds) at time $t$ in the bounded coalescent starting with $n$ plants and $m$ seeds.
 The block-counting process of the bounded coalescent with parameters $c_1,c_2>0$ has the following transition rates. For $ i\leq n$ and $j\leq m$,
\[
(\bar{N}_{n,m}(t),\bar{M}_{n,m}(t))\text{ jumps from }(i,j)
\text{ to}\left\{ \begin{array}{lll}
(i-1,j), & \textrm{at rate } \binom{i}{2}+c_1i{\bf 1}_{\{j= m\}},\\
(i-1,j+1), &  \textrm{at rate } c_1i{\bf 1}_{\{j<m\}},\\
(i+1,j-1), & \textrm{at rate } c_2j.\\
\end{array}
\right .
\]

By coupling the seed bank coalescent with its bounded version, we obtain a lower bound for $\theta_n$ and an upper bound for $N_n(\theta_n)$.
\begin{lemma}\label{theta}
Recall $T$ and $Z$ from Proposition \ref{summarythetan}. We have that
\begin{equation}\label{borneinfTh}
 \lim_{n\rightarrow\infty}\mathbb{P}\left( {\theta_n}\log n\leq t\right)\leq\P(T\leq t)
\end{equation}
and
\begin{equation}\label{eq9}
\lim_{n\rightarrow\infty} \mathbb{P}\left( {N_n({\theta_n}})>z{\log n}\right)\leq\mathbb{P}\left( Z>z\right).
\end{equation}
\end{lemma}
\begin{proof}
Fix $\varepsilon>0$ and denote $\lfloor2c_1(1+\varepsilon)\log n\rfloor$ by $m_n$.
On the event $\{M_{n}({\theta_n})\leq m_n\}$, which occurs asymptotically with probability 1 by Lemma \ref{cotasupM}, the variable $\theta_n$ is bounded from below, stochastically, by the random variable $\bar\theta_{n,m_n}$ defined by
  $$\bar\theta_{n,m_n}=\inf\{t\geq0:\bar M_{n,m_n}(t-)>\bar M_{n,m_n}(t) \}$$
  and having exponential distribution with parameter $c_2m_n$.
  Then, for $t>0$
\begin{eqnarray*}
\P\left(\theta_n\log n\leq{t} \right)&=&\P \left(\theta_n\log n\leq{t}, M_{n}({\theta_n})\leq m_n \right)+o(1)\\
&\leq& \P \left(\bar\theta_{n,m_n}\log n\leq{t}\right)+o(1)\\
&=&1-\exp\left\{-t\frac{c_2\lfloor2c_1(1+\varepsilon)\log n\rfloor}{\log n}\right\}+o(1).
\end{eqnarray*}  
So, for any $\varepsilon>0$,
\begin{equation}\label{cotaprobthetasup}
\lim_{n\rightarrow\infty}\P\left(\theta_n\log n\leq{t} \right)\leq\P\left(T\leq t(1+\varepsilon)\right).
\end{equation}
This gives \eqref{borneinfTh}.

To prove \eqref{eq9}, observe that, on the event $\{M_{n}({\theta_n})\leq m_n\}$, the variable $N_n(\theta_n)$ is bounded from above, stochastically, by the random variable $\bar N_{n,m_n}(\bar\theta_{n,m_n})$.
So,
\begin{equation}\label{N<barN}
\mathbb{P}\left( {N_n({\theta_n}})>z{\log n}\right)\leq\mathbb{P}\left(  {\bar N_{n,m_n}({\bar\theta_{n,m_n}}})>z{\log n}\right)+\P(M_{n}({\theta_n})> m_n).
\end{equation}
 Let us study the asymptotic of $\bar N_{n,m_n}(\bar\theta_{n,m_n})$.
 To this, we have that
 \begin{eqnarray*}
\mathbb{P}(\bar N_{n,m_n}(\bar\theta_{n,m_n})\leq z\log n)&=&\prod_{i=\lfloor z\log n\rfloor+1}^n\frac{\binom{i}{2}+c_1i}{\binom{i}{2}+c_1i+c_2m_n}\\
&=&\exp\left\{ -\sum_{i=\lfloor z\log n\rfloor+1}^n\log\left( 1+\frac{2c_2m_n}{i(i-1+2c_1)}\right)  \right\}\\
&\sim &\exp\left\{ -2c_2m_n\sum_{i=\lfloor z\log n\rfloor+1}^n\frac{1}{i^2}\right\}.
\end{eqnarray*}
By a Riemann sum argument, we know that
\begin{equation}\label{Riemann1/z}
\lim_{n\rightarrow\infty}m_n\sum_{i=\lfloor z\log n\rfloor+1}^n\frac{1}{i^2}=2c_1(1+\varepsilon)\int_z^\infty\frac{1}{x^2}dx=\frac{2c_1(1+\varepsilon)}{z}.
\end{equation}
Since $\P(Z\leq z)=\exp\{-{4c_1c_2}/{z}\}$, we obtain, by taking the limits in \eqref{N<barN}, that
$$
\lim_{n\rightarrow\infty} \mathbb{P}\left( {N_n({\theta_n}})>z{\log n}\right)\leq\mathbb{P}\left( Z>z/(1+\varepsilon)\right)
$$
which implies \eqref{eq9}.
\end{proof}

The bounded seed bank coalescent is also useful to bound $N_n(t)$ from above, for any $t\geq0$.
Let $(K_n(t))_{t\geq0}$ stand for the block-counting process of the Kingman coalescent starting with $n$ lineages.
Let $(\chi_i(t))_{i\geq 1}$ be a sequence of i.i.d. Bernoulli variables of parameter $1-\exp(-c_2t)$. Those variables are more easily understood as $\chi_i(t)={\bf 1}_{\{{e}_i<c_2t\}}$ where the ${e}_i$'s are i.i.d. standard exponential variables.
It is easy to convince oneself that, on the event $\{\sup_{t\geq 0} M_n(t)\leq m\}$, stochastically, 
\begin{equation}\label{boundNnt}
N_n(t)\leq K_n(t)+\sum_{i=1}^{m}\chi_i(t).
\end{equation}
 This follows because $K_n(t)$ bounds the number of blocks that have not deactivate before time $t$ and $\sum_{i=1}^{m}\chi_i(t)$ bounds the number of blocks that have already reactivated. Both processes are independent.

We now prove a useful lemma thanks to the two couplings introduced previously.
To simplify the notations here and in the sequel, denote
$\tau_n^{\lfloor (\log n)^a\rfloor}
$ by $\tau_n^{(a)}$, for any $a>0$.

\begin{lemma}\label{desviaciontau}
For $a>b\geq0$ such that $a+b>1$,
\begin{equation*}
\lim_{n\rightarrow\infty} \mathbb P\left(\tau_n^{(a)}\leq{(\log n)^{-b}}\right)=1.
\end{equation*}
\end{lemma}
\begin{proof}
Denote $\lfloor 2c_1(1+\varepsilon)\log n\rfloor$ by $m_n$ and let $E_n=\{\sup_tM_n({t})\leq m_n\}$.
%and let $E_{n}^{(a)}=\bigcap_{i=1}^{\lfloor n-(\log n)^a\rfloor}E_{n,i}$
%and let  $E_n=\{\sup_t M_n(t)<2c(1+\varepsilon)\log n\}$.
We start by observing that
$$\mathbb P(\tau_n^{(a)}>{(\log n)^{-b}})=\mathbb P(\tau_n^{(a)}>{(\log n)^{-b}}, E_n)+\mathbb P(\tau_n^{(a)}>{(\log n)^{-b}},  E_n^c)$$
From \eqref{supM}, we get that
\begin{equation*}
\mathbb P( E_n^c)\leq\frac{1}{2c_1\varepsilon^2\log n}.
\end{equation*}
So it just remains to control the probability on the event $E_n$.
Recall $(K_n(t))_{t\geq0}$ and $(\chi_i(t))_{i\geq 1}$ from \eqref{boundNnt}.
   Let $\omega_{n,a}=\inf\{t>0:K_n(t)=\frac{1}{2}(\log n)^{a}\}$. 
Observe that
\begin{eqnarray*}
\{\tau_n^{(a)}>t,E_n\}
&=&\{N_n(t)>(\log n)^{a},E_n\}\\
&\subset&\{K_n(t)+\sum_{i=1}^{m_n} \chi_i(t)>(\log n)^{a}\}\\
&\subset&\{K_n(t)>\frac{1}{2}(\log n)^{a}\}\cup\{\sum_{i=1}^{m_n} \chi_i(t)>\frac{1}{2}(\log n)^{a}\}\\
&=&\{\omega_{n,a}>t\}\cup\{\sum_{i=1}^{m_n} \chi_i(t)>\frac{1}{2}(\log n)^{a}\}.
\end{eqnarray*}

Taking $t=(\log n)^{-b}$, we obtain
$$ \mathbb{P}(\tau_n^{(a)}>(\log n)^{-b}, E_n)
\leq \mathbb{P}(\omega_{n,a}> (\log n)^{-b})+ \mathbb P( \sum_{i=1}^{m_n} \chi_i((\log n )^{-b})>\frac{1}{2}(\log n)^{a} ).$$
An elementary calculation on sum of independent exponential variables shows that $$\E[\omega_{n,a}]\sim4(\log n)^{-a}.$$ %\textcolor{red}{and $\Var(\omega_{n,a})\sim\frac{32}{3}(\log n)^{-3a}$.}
So, Markov's inequality for $\omega_{n,a}$ gives
$$\mathbb{P}(\omega_{n,a}> (\log n)^{-b})\leq C(\log n)^{b-a}$$
for some constant $C>0$, which converges to 0 whenever $b<a$.
On the other hand, Markov's inequality applied to a binomial random variable with parameters $\lfloor2c_1(1+\varepsilon)\log n\rfloor$ and $1-\exp(-c_2(\log n)^{-b})$ (which expectation is of order $(\log n)^{1-b}$) lead to
$$\mathbb P( \sum_{i=1}^{m_n} \chi_i((\log n )^{-b})>\frac{1}{2}(\log n)^{a} )\leq C(\log n)^{1-b-a}.$$
This quantity converges to 0 as $a+b>1$. 
\end{proof}

\begin{rem} The rate of coalescence is quadratic with respect to the number of plants while the rate of deactivation (resp.  the rate of activation) is linear with respect to the number of plants (resp. the number of seeds). 
The latter lemma suggests that, until time $\tau_n^{(a)}$, for $a>1/2$, the block-counting process $(N_n(t))_{t\ge0}$ behaves similar to that of the Kingman coalescent.
However, at time $\tau_n^{(1/2)}$, the system reaches a level of $\sqrt{\log n}$ plants and the times of decay are no longer close to those of the Kingman coalescent.
Indeed, at this time, we claim that the number of seeds is of order $\log n$ and the coalescence events do not dominate any more the dynamics. The seed bank coalescent then enters into a mixed regime with coalescence and activation occurring at the same velocity.
\end{rem}

%The latter lemma permits to complement Lemma \ref{theta} to obtain the following result.
%\begin{prop}\label{cotasupN}
%For any $\varepsilon>0$,
%\begin{equation*}
% \lim_{n\rightarrow\infty}\mathbb{P}\left({\theta_n}>\tau_n^{(1+\varepsilon)}\right)=1,
%\end{equation*}
%which is equivalent to 
%\begin{equation*}
% \lim_{n\rightarrow\infty}\mathbb{P}\left( N_n({\theta_n})<{(\log n)^{1+\varepsilon}}\right)=1
%\end{equation*}
%since the process $(N_n(t))_{t\geq0}$ is non-increasing until time $\theta_n$.
%\end{prop}
%%\begin{proof}
%Let $1<\delta<a$.
%We have
%\begin{eqnarray*}
%\P(\theta_n<\tau_n^{(a)})&=&\P\left( \theta_n<\tau_n^{(1+\varepsilon)}<(\log n)^{-\delta} \right) +\P\left( \theta_n<\tau_n^{(1+\varepsilon)}, \tau_n^{(1+\varepsilon)}\geq(\log n)^{-\delta} \right)\\
%&\leq& \P\left( \theta_n<{(\log n)^{-\delta}}\right) +\P\left( \tau_n^{(1+\varepsilon)}\geq {(\log n)^{-\delta}}\right).
%\end{eqnarray*}

%The first term converges to 0 thanks to Lemma \ref{theta}.
%The second term also converges to 0 thanks to  Lemma \ref{desviaciontau}.
%\end{proof}

We now provide the lower bound for $M_n(\theta_n)$.
This result, combined with Lemma \ref{cotasupM} provides the convergence \eqref{limitM} in Proposition \ref{summarythetan}.

\begin{lemma}\label{cotainfM}
For any $\varepsilon>0$ and $a>1$,
\begin{equation}\label{borneinfMtau}
\lim_{n\rightarrow\infty}\P(M_n({\tau_n^{(a)}})>2c_1(1-\varepsilon)\log n)=1.
\end{equation}
which implies that
\begin{equation}\label{borneinfMtheta}
\lim_{n\rightarrow\infty}\P(M_n({\theta_n})>2c_1(1-\varepsilon)\log n)=1.
\end{equation}
\end{lemma}

\begin{proof}
Let us first note that \eqref{eq9} implies that
$$\lim_{n\rightarrow\infty}\P(N_n({\theta_n})<(\log n)^a)=1,$$
which, thanks to the monotonicity of $(N_n(t))_{t\geq0}$ until time $\theta_n$, is equivalent to
\begin{equation*}
 \lim_{n\rightarrow\infty}\mathbb{P}\left({\theta_n}>\tau_n^{(a)}\right)=1.
\end{equation*}
Due to the monotonicity of $(M_n(t))_{t\geq0}$ until time $\theta_n$, \eqref{borneinfMtau} implies \eqref{borneinfMtheta}.

Now, on the event $\{\theta_n>\tau_n^{(a)}\}$, we have
$$M_n(\tau_n^{(a)})=\sum_{i=\lfloor(\log n)^a\rfloor}^{n-1}B^i_{n}$$
where the $B_n^i$'s are the Bernoulli random variables introduced in \eqref{beeni}.
So,
\begin{align*}
\P(M_n({\tau_n^{(a)}})<2c_1(1-\varepsilon)\log n)&=\P( M_n({\tau_n^{(a)}})<2c_1(1-\varepsilon)\log n, \theta_n>\tau_n^{(a)})+o(1)\\ &\leq\P\left(\sum_{i=\lfloor (\log n)^a\rfloor}^{n-1}B^i_{n}<2c_1(1-\varepsilon)\log n\right)+o(1)\\
&=\P\left(\sum_{i=1}^{n-1}B^i_{n}<2c_1(1-\varepsilon)\log n+\sum_{i=1}^{\lfloor (\log n)^a\rfloor-1}B^i_{n}\right)+o(1).
\end{align*}
It is easy to convince oneself that $\sum_{i=1}^{\lfloor (\log n)^a\rfloor-1}B^i_{n}$ is of order $\log(\log n)^{a}$. The latter converges to 0 thanks to Bienaym\'e-Chebyshev's inequality.
\end{proof}

We 
are now able to end the overview of the system at time $\theta_n$.
The following result, combined with Lemma \ref{theta} provides the convergences \eqref{limitN} and \eqref{limittheta} in Proposition \ref{summarythetan}.
\begin{lemma}\label{cotainfN}
Recall $T$ and $Z$ from Proposition \ref{summarythetan}. We have that
\begin{equation}\label{Frechetinf}
\lim_{n\rightarrow\infty} \mathbb{P}\left( {N_n({\theta_n}})\leq z{\log n}\right)\leq\mathbb{P}\left( Z\leq z\right).
\end{equation}
which implies that
\begin{equation}\label{cotasuptheta}
\lim_{n\rightarrow\infty}\mathbb{P}\left(\theta_n\log n >t\right)\leq\P(T>t).
\end{equation}

\end{lemma}

\begin{proof}
Fix $\varepsilon>0$ and denote $\lfloor 2c_1(1-\varepsilon){\log n}\rfloor$ by $m_n$. Also denote $\tau_n^{\lfloor z\log n\rfloor}$ by $\hat\tau_n$.
First observe that
$$\mathbb{P}\left( {N_n({\theta_n}})\leq z\log n\right)=\P(\theta_n\geq\hat\tau_n)$$
So it is enough proving that
\begin{equation}\label{upboundtheta}
\lim_{n\rightarrow\infty}\P(\theta_n\geq\hat\tau_n)\leq \P(Z\leq z).
\end{equation}

For any $t\geq0$, define $X({t})$ to be the number of reactivations until time $t$. Let $\mathcal E_i$ be an exponential random variable with parameter $c_2i$, that can be understood as the minimum of $i$ independent exponential random variables with parameter $c_2$. Then,  for any $a>1$,

\begin{eqnarray*}
\mathbb{P}(\theta_n\geq\hat\tau_n)&=&\mathbb{P}(X({\hat\tau_n})=0)=\mathbb{P}(X({\hat\tau_n})-X({\tau_n^{(a)}})=0,X({\tau_n^{(a)}})=0)\\
&\leq&\mathbb{P}(X({\hat\tau_n})-X({\tau_n^{(a)}})=0\mid X({\tau_n^{(a)}})=0)\\
&\leq&\mathbb{P}(\mathcal E_{M_{n}({\tau_n^{(a)}})}>\hat\tau_n-\tau_n^{(a)}).
\end{eqnarray*}
The  latter inequality follows by observing that if there are no activations in the time interval $[\tau^{(a)}_n,\hat\tau_n]$, then none of the $M_{n}({\tau^{(a)}_n})$ seeds present at time $\tau^{(a)}_n$ have activated. 
Hence,
\begin{align*}
\mathbb{P}(\theta_n\geq\hat\tau_n)
&\leq\mathbb{E}\left[ e^{-c_2(\hat\tau_n-\tau^{(a)}_n) M_n({\tau^{(a)}_n})} \right]\\
&=\mathbb{E}\left[ e^{-c_2(\hat\tau_n-\tau^{(a)}_n)M_n({\tau^{(a)}_n})}{\bf1}_{\{M_n({\tau^{(a)}_n})>m_n\}} \right]\\
&\hspace{.4cm}+ \mathbb{E}\left[ e^{-c_2(\hat\tau_n-\tau^{(a)}_n)  M_n({\tau^{(a)}_n})}{\bf 1}_{\{M_n({\tau^{(a)}_n})\leq m_n\}} \right]\\
&\leq\mathbb{E}\left[ e^{-c_2m_n(\hat\tau_n-\tau^{(a)}_n) } \right]
+\mathbb{P}(M_n({\tau^{(a)}_n})\leq m_n).
\end{align*}
So, by denoting for simplicity ${n}_z=\lfloor z\log n\rfloor$ and $n_a=\lfloor (\log n)^{a}\rfloor$, and by \eqref{tau>bartau}, we obtain
\begin{equation}\label{th<taub}
\mathbb{P}(\theta_n\geq\hat\tau_n)
\leq \mathbb{E}\left[ e^{-c_2m_n\sum_{i={n}_z+1}^{n_a}(\underline{\tau}^{i-1}_n-\underline{\tau}_n^{i})} \right]
+\mathbb{P}(M_n({\tau^{(a)}_n})\leq m_n).
\end{equation}

Since the variables $\underline{\tau}_n^{i-1}-\underline{\tau}_n^{i}$ are independent and exponentially distributed, we have
\begin{align*}
\mathbb{E}\left[ e^{-c_2m_n\sum_{i=n_z+1}^{n_a}(\underline{\tau}^{i-1}_n-\underline{\tau}_n^{i})}\right]
&=\prod_{i={n}_z+1}^{{n}_a}\frac{\binom{i}{2}+c_1i}{\binom{i}{2}+c_1i+c_2m_n}\\
&=\exp\left\{ -\sum_{i={n}_z+1}^{n_a}\log\left( 1+\frac{2c_2m_n}{i(i-1+2c_1)}\right)  \right\}.
\end{align*}
Now, we can use equivalences.
$$\mathbb{E}\left[ e^{-c_2m_n\sum_{i=n_z+1}^{n_a}(\underline{\tau}^{i-1}_n-\underline{\tau}_n^{i})}\right]
\sim\exp\left\{ -\sum_{i={n}_z+1}^{n_a}\frac{2c_2m_n}{i^2}  \right\}
$$
A similar limit as that given in \eqref{Riemann1/z} implies that
\begin{equation}\label{term1}
\lim_{n\rightarrow\infty}\mathbb{E}\left[ e^{-c_2m_n\sum_{i={n}_z+1}^{n_a}(\underline{\tau}^{i-1}_n-\underline{\tau}_n^{i})}\ \right]=e^{-\frac{4c_1c_2(1-\varepsilon)}{z}}=\P(Z\leq z/(1-\varepsilon)).
\end{equation}
Plugging \eqref{term1} and \eqref{borneinfMtau} into \eqref{th<taub}, and observing that the result is true for any $\varepsilon>0$, we get \eqref{upboundtheta}.

A very similar path is followed to obtain \eqref{cotasuptheta}.
For some $t>0,$ let $t_n=t(\log n)^{-1}$ and for some $b>1$, let $s_n=(\log n)^{-b}$. As before, we get
\begin{align*}
\mathbb{P}\left(\theta_n\log n >t\right)&=
\mathbb{P}\left( \theta_n>t_n\right) \\
&=\mathbb{P}\left( X\left( t_n\right) =0\right)\\
&\leq e^{-c_2m_n(t_n-s_n) } 
+\mathbb{P}(M_n(s_n)\leq m_n),
\end{align*}
The first term converges to $\P(T>t(1-\varepsilon))$ and the second to 0.
To get the latter, first use (\ref{borneinfTh})
 to see that $$\lim_{n\rightarrow\infty}\P(\theta_n>s_n)=1.$$
Then, just choose $a>b$ such that Lemma \ref{desviaciontau} holds, and use \eqref{borneinfMtau}
Since the result is true for any $\varepsilon>0$, we get \eqref{cotasuptheta}.
%To obtain \eqref{cotasuptheta}, fix $\varepsilon>0$ and choose $a_1\in(1/2,1)$ such that $1-\varepsilon<a_1$, then
%\begin{eqnarray*}
%\P\left( \theta_n\geq{(\log n)^{-1+\varepsilon}}\right)&=&\P\left(  \theta_n\geq{\log n^{-1+\varepsilon}}, \theta_n>\tau_n^{(a_1)} \right)\\
%&+&\P\left(  \theta_n\geq{(\log n)^{-1+\varepsilon}}, \theta_n\leq\tau_n^{(a_1)} \right) \\
%&\leq& \P\left( \theta_n\geq\tau_n^{(a_1)}\right) +\P\left( \tau_n^{(a_1)}\geq {(\log n)^{-1+\varepsilon}}\right).
%\end{eqnarray*}
%
%By Lemma \ref{desviaciontau}, we get that the second term above converges to 0 as $n$ goes to infinity.
%So it is enough proving that for every $a_1\in\left( {1}/{2},1\right)$, 
%\begin{equation}\label{belowboundtheta}
%\lim_{n\rightarrow\infty}\P(\theta_n<\tau_n^{(a_1)})= 1,
%\end{equation}
%which can be obtained as in the first part of the proof,
%to get \eqref{upboundtheta}.
\end{proof}

\section{Branch Lengths}\label{secLength}
In this section, we study the total branch length $L_n$ of the seed bank coalescent starting with $n$ plants and no seeds as defined in \eqref{defL} and prove Theorem \ref{thmL} by combining upcoming Theorems \ref{thmA} and \ref{thmI}.

\subsection{The active length}
Consider the active  length defined in \eqref{defA}.
We prove that this variable has the same asymptotics as those of the total length of the Kingman coalescent.
\begin{thm}\label{thmA}
Consider the seed bank coalescent starting with $n$ plants and no seeds. Then,
$$\lim_{n\rightarrow\infty}\frac{A_n}{\log n}=2$$
in probability.
\end{thm}
\begin{proof}
Recall the notation $\tau_n^{(a)}=\tau_n^{\lfloor ({\log n})^a\rfloor}$ and consider some $a\in(1/2,1)$.
Divide $A_n$ in three parts
$$A^1_n=\int_0^{\theta_n}N_n(t)dt\quad, \quad A^2_n=\int_{\theta_n}^{\tau_n^{(a)}}N_n(t)dt\quad \text{ and }\quad A^3_n=\int_{\tau_n^{(a)}}^{\sigma_n}N_n(t)dt.$$
Here we will work on the event $\{\theta_n\leq\tau_n^{(a)}\}$.
On the complementary event, the proof is more easily following the same steps.
The result is obtained from \eqref{cvA1}, \eqref{cvA2} and \eqref{cvA3} in the sequel.

i) Let us first prove that
\begin{equation}\label{cvA1}
\lim_{n\rightarrow\infty}\frac{A^1_n}{\log n}=2
\end{equation}
in probability.
Observe that, between times 0 and $\theta_n$, only coalescence or deactivation events occur. This implies that we can rewrite  $A_n^1$ as follows, 
\begin{align*}
A_n^1=\sum_{i=N_n({\theta_n})+1}^n iE_i,
\end{align*}
where, given $M_n(\tau_n^{i})$,  the $E_i$'s are independent exponential random variables with respective parameters $\binom{i}{2}+c_1i+c_2M_n(\tau_n^{i})$. 
Let $h_n=\sum_{i=1}^{n-1}\frac{2}{i+2c_1}$. By proving that
$$\mathbb E[|A_n^1-h_n|]=o(\log n),$$
we get the desired result.
To this. Observe that the variable $A_n^1$ is stochastically bounded by the length of a Kingman coalescent with freezing, that is
$$H_n=\sum_{i=2}^{n} i V_i,
$$
where  the $V_i$'s, as in Section \ref{secgamma}, are independent exponential random variables with respective parameters $\binom{i}{2}+c_1i$.
This is true because the seeds ``accelerate'' the jump times.
To be precise consider the following coupling. Let $V_i=\min{\{ E_i^{(c)}, E_i^{(d)}\}}$ where $ E_i^{(c)}$ is exponential with parameter $\binom{i}{2}$ and $ E_i^{(d)}$ is exponential with parameter $c_1i$. Now let $ E_{i,m}^{(a)}$ be exponential with parameter $c_2m$. 
Construct a process $(\tilde N_n(t),\tilde M_n(t))_{t\geq0}$, equal in distribution to $(N_n(t),M_n(t))_{t\geq0}$ up to time $\theta_n$, recursively, using these exponential random variables. This is,
\[
(\tilde{N}_{n}(t),\tilde{M}_{n}(t))\text{ jumps from }(i,m)
\text{ to}\left\{ \begin{array}{lll}
(i-1,m), & \textrm{if } \min{\{ E_i^{(c)}, E_i^{(d)},  E_{i,m}^{(a)}\}}= E_i^{(c)},\\
(i-1,m+1), &  \textrm{if } \min{\{ E_i^{(c)}, E_i^{(d)},  E_{i,m}^{(a)}\}}= E_i^{(d)},\\
(0,0), & \textrm{otherwise }.\\
\end{array}
\right .
\]
Here $(0,0)$ represents a cemetery state. Note that in distribution $(\tilde N_n(t),\tilde M_n(t))=(N_n(t),M_n(t))1_{\{\theta_n>t\}}$. Thus, by writing $(\tilde\tau_n^i)_{i=1}^{n}$ for the successive jump times of the new process and $\tilde r_n=\sup\{i\geq1:\min{\{ E_i^{(c)}, E_i^{(d)},  E_{i,\tilde M_n(\tilde\tau_n^{i})}^{(a)}\}}= E_{i,\tilde M_n(\tilde\tau_n^{i})}^{(a)}\}$, we obtain that
$$
A_n^1= \sum_{i=\tilde r_n+1}^n i V_i\leq \sum_{i=2}^{n} i V_i=H_n,
$$
where the first equality is in distribution and the others stand almost surely. 
The first equality is true because, although the  $V_i$'s are variables with the ``wrong" parameter, they are not independent of $\tilde r_n$, and this dependence ``accelerates" these exponential random variables. 
Hence,
$$
\mathbb E[|A_n^1-h_n|]\leq\mathbb E[H_n-A_n^1]+\mathbb E[|H_n-h_n|].
$$
The second term is bounded thanks to the $L^1$-convergence of sums of independent exponential variables.
For the first term,
\begin{align*}
\mathbb E[H_n-A_n^1]
&=\mathbb{E}\left[ H_n- \mathbb{E}\left[A^1_n| N_n({\theta_n}),(M_n(\tau_n^{i}))_{i\geq1}\right]\right] \\
&=h_n-\mathbb{E}\left[\sum_{i=N_n({\theta_n})+1}^n \frac{2}{i-1+2c_1+\frac{2c_2M_n(\tau_n^{i})}{i}}\right] \\
&\leq h_n-\mathbb{E}\left[ \sum_{i=N_n({\theta_n})+1}^{n} \dfrac{2}{i-1+2c_1+\frac{2c_2 \sup_t M_n(t)}{i}}\right].
\end{align*}
Then, denote $\lfloor 2c_1(1+\varepsilon_1)\log n\rfloor$ by $m_n$,
and $\lfloor(\log n)^{1+\varepsilon_2}\rfloor$ by $a_n$, for some $\varepsilon_1,\varepsilon_2>0$. 
Now, set the event $E_n=\{\sup_tM_n({t})\leq  m_n, N_n({\theta_n})\leq a_n\}$.
We obtain that
\begin{align*}\label{ciespA1}
\mathbb E[H_n-A_n^1]
&\leq h_n-
\mathbb{E}\left[{\bf1}_{E_n} \sum_{i=N_n({\theta_n})+1}^{n} \dfrac{2}{i-1+2c_1+\frac{2c_2 \sup_t M_n(t)}{i}}\right]\\
&\leq h_n-
\mathbb{P}(E_n) \sum_{i=a_n+1}^{n} \frac{2}{i-1+2c_1+\frac{2c_2m_n}{i}}\notag\\
&\leq h_n-
\mathbb{P}(E_n) \sum_{i=a_n+1}^{n} \frac{2}{i-1+2c_1+\frac{2c_2m_n}{a_n+1}}.\notag
\end{align*}
Since $\frac{m_n}{a_n+1}\leq C(\log n)^{-\varepsilon_2}$  for some constant $C$ and $\mathbb{P}(E_n)$ converges to 1 (thanks to Proposition \ref{summarythetan}), we get that
\begin{equation*}
\mathbb E[H_n-A_n^1]
=o(\log n).
\end{equation*}
The $L^1$-convergence is thus obtained. This implies \eqref{cvA1}.

ii) Let us now prove that 
\begin{equation}\label{cvA2}
\lim_{n\rightarrow\infty}\frac{A^2_n}{\log n}=0
\end{equation}
in probability.
It is clear that, almost surely,
$$A^2_n\leq\tau_n^{(a)}(N_n({\theta_n})+M_n({\theta_n})).$$
Combining Proposition \ref{summarythetan} and Lemma \ref{desviaciontau} (choosing $b<a$), we obtain the result.

iii) Finally, let us prove that
\begin{equation}\label{cvA3}
\lim_{n\rightarrow\infty}\frac{A^3_n}{\log n}=0
\end{equation}
in probability.
To this end, define $U_0=N_n({\tau_n^{(a)}})=\lfloor(\log n)^a\rfloor$ (by definition), $V_0=M_n({\tau_n^{(a)}})$ (which, by Lemma \ref{cotasupM}, is stochastically bounded by $2c_1(1+\varepsilon)\log n$), and, for any $k\geq1$, $U_k$ (resp. $V_k$) as the number of plants (resp. seeds) at the $k$th event after time $\tau_n^{(a)}$. Each event can be a coalescence, an activation or a deactivation.
 Note that the increments of $U_k$ and $V_k$ are in $\{-1,1\}$.
Let $S_n$ be the number of jump times during the interval $(\tau_n^{(a)},\sigma_n]$, i.e.
$$S_n=\inf\lbrace k\geq1:U_k+V_k=1\rbrace.$$
With these notations, the active branch length on this time interval can be written as %can be approximated by
\[A_n^3=\sum_{k=0}^{S_n-1}U_kE_k\]
where, conditional on $U_k$ and $V_k$, the $E_k$'s are independent exponential random variables with respective parameters $\binom{U_k}{2}+c_1U_k+c_2V_k$. So, we have 
%\begin{equation}\label{espA3}
$$\mathbb{E}[A_n^{3}]=\mathbb{E}\left[ \sum_{k=0}^{S_n-1}\frac{U_k}{\binom{U_k}{2}+c_1U_k+c_2V_k}\right].
$$
Now define 
$$D_n:=|\{k\geq0:U_{k+1}-U_k=-1,V_{k+1}-V_k=1\}|$$
as the number of deactivations during this time interval, and observe that
$$\mathbb{E}[D_n]=\mathbb{E}\left[ \sum_{k=0}^{S_n-1}\frac{c_1U_k}{\binom{U_k}{2}+c_1U_k+c_2V_k}\right].$$
This implies that
$$\mathbb{E}[A_n^{3}]=\frac1c_1\mathbb{E}[D_n].$$
So, it is enough to study the expectation of $D_n$.
We decompose 
$$D_n=\sum_{i=2}^{N_n({\tau_n^{(a)}})+M_n({\tau_n^{(a)})}}D_n^{i}$$
where $D_n^{i}$  is the number of deactivations occurring while the total number of lineages equals $i$, that is, $D_n^{i}:=|\{k\geq0:U_{k+1}-U_k=-1,V_{k+1}-V_k=1, U_k+V_k=i\}|$. We will bound $\mathbb E[D_n]$ thanks to the next model from Definition 4.9 of \cite{Blath2016}.

Let $(\widehat{N}_n(t),\widehat{M}_n(t))_{t\geq0}$ having the same transitions as $(N_n(t),M_n(t))_{t\geq0}$ whenever $\widehat N_n(t)\geq\sqrt{\widehat N_n(t)+\widehat M_n(t)}$. If not, coalescence events are not permitted.
For any $i\geq2$, by Lemma 4.10 of \cite{Blath2016}, $\E[D_n^i]\leq \E[\widehat D_n^i]$, where $\widehat D_n^i$ stands for the number of deactivations in this model while $\widehat N_n(t)+\widehat M_n(t)=i$.
In what follows we will give an idea of why $\mathbb E[\widehat D_n^i]=O(i^{-1/2})$, implying that $\mathbb E[ D_n]=O((\log n)^{1/2})$, and hence proving \eqref{cvA3}.

 Details of the proof, which are unfortunately quite tedious, can be found inside the proof of Lemmas 4.10 and 4.11 of \cite{Blath2016}.
 In the sequel, suppose that $c_1=c_2=1$, for sake of simplicity.

Fix $i\geq 2$. The higher values that $\widehat D_n^i$ can take is when the coalescences are not permitted.
Thus suppose that at time $t$, $\widehat N_n(t)+\widehat M_n(t)$ reaches $i$, with $\widehat N_n(t-)=\lfloor\sqrt {i}\rfloor+1\geq\sqrt {i+1}$. This means that $\widehat N_n(t)=\lfloor\sqrt i\rfloor\leq\sqrt i$. Reactivations are then needed to allow a new coalescence.
Conditional on this configuration, the probability that $\widehat D_n^i$ equals 0 is equivalent to
$$\frac{i-\lfloor\sqrt i\rfloor}{i}\times\frac{\binom{\lfloor\sqrt i\rfloor}{2}}{\binom{\lfloor\sqrt i\rfloor}{2}+\lfloor\sqrt i\rfloor}\sim1-\frac{3}{\sqrt i}=:p_i.$$
This corresponds approximately to the probability of one reactivation, followed by one coalescence before one deactivation.
So we have the following almost sure bound
$$\widehat D_n^i\leq\sum_{j=0}^{G^i-1}\Delta_j$$
where $G^i$ is a geometric random variable of parameter $p_i$ and the $\Delta_j$'s give the number of deactivations between each visit of the state $\lfloor\sqrt i\rfloor$.
The time when coalescence is not allowed, is stochastically bounded from above by the time that a random walk that goes up one unit at rate $i-\sqrt{i}$ (rate at of a reactivation) and down at rate $\sqrt{i}$ (rate of a deactivation), started at zero, spends below level $\sqrt i$. The random walk has ballistic speed of order $i$. In particular, it reaches the level $\sqrt{i}$ after $\sqrt{i}/i=1/\sqrt{i}$ units of time in average. During the period in which coalescence events are not allowed there are always less that $\sqrt{i}$ plants, each of which deactivates at rate $c_1(=1)$. Then, we conclude that, for any $j$, 
$$\mathbb E[\Delta_j]\leq\frac{1}{\sqrt{i}}\cdot \sqrt{i}=1$$
This uniform bound implies that 
{$$\E[\widehat D_n^i]\leq \E[G^i-1]\E[\Delta_1]=O\left(\frac{1}{\sqrt i}\right),$$ }
 since $\mathbb E[G^i-1]\sim\frac{3}{\sqrt i}$.
\end{proof}

%\begin{figure}
%\centering
%\begin{tikzpicture}[scale=0.5]
 % \draw (0,2.5) circle(0.5);
 % \node at (0,0.5) {$\times$};
  %\node at (0,5) {$\times$};
  %\node at (0,3) {$\times$};
  %\node at (0,8) {$\times$};
  %\node at (1,2) {$\times$};
  %\node at (2,1) {$\times$};
  %\node at (2,7) {$\times$};
  %\node at (4,6) {$\times$};
  %\node at (6,2.5) {$\times$};
  %\node at (6,5.5) {$\times$};
  %\node at (4,9) {$\times$};
  %\node at (1,10) {$\times$};
  %\node at (5,11) {$\times$};
  %\node at (-2,4)  {$\Delta_1=5$};
   %\node at (-2,10)  {$\Delta_2=1$};
   %%\node at (12,3)  {$\lfloor\sqrt{i}\rfloor=3$};
   %\node at (12,4)  {${i}=10$};
   %\node at (12,13)  {${i}=9$};
%-----------------------------------
  %\draw[thick] (0,0) -- (0,0.5);
  %\draw[dashed] (0,0.5) -- (0,3);
  %\draw[thick] (0,3) -- (0,5);
  %\draw[dashed] (0,5) -- (0,8);
   %\draw[thick] (0,8) -- (0,12);
   %\draw[thick] (-3,8) -- (12,8);
  %\draw[thick] (1,0) -- (1,2);
   % \draw[dashed] (1,2) -- (1,10);
   % \draw[thick] (1,10) -- (1,12);
    % \draw[thick] (-3,12) -- (12,12);
     % \draw[thick] (0.5,12) -- (0.5,13);
  %\draw[thick] (2,0) -- (2,1);
   % \draw[dashed] (2,1) -- (2,7);
    %\draw[thick] (2,7) -- (2,13);
  %\draw[dashed] (3,0) -- (3,13);
  %\draw[dashed] (4,0) -- (4,6);
  %\draw[thick] (4,6) -- (4,9);
   %\draw[dashed] (4,9) -- (4,13);
  %\draw[dashed] (5,0) -- (5,11);
  % \draw[thick] (5,11) -- (5,13);
  %\draw[dashed] (6,0) -- (6,2.5);
  %\draw[thick] (6,2.5) -- (6,5.5);
  %\draw[dashed] (6,5.5) -- (6,13);
  %\draw[dashed] (7,0) -- (7,13);
  %\draw[dashed] (8,0) -- (8,13);
  %\draw[dashed] (9,0) -- (9,13);
%\end{tikzpicture}

%\caption{}
%\label{delta_realization}
%\end{figure}

\subsection{The inactive length}
Consider the inactive  length defined in \eqref{defI}.
\begin{thm}\label{thmI}
Consider the seed bank coalescent starting with $n$ plants and no seeds. Then,
$$\lim_{n\rightarrow\infty}\frac{I_n}{\log n}=\frac{2c_1}{c_2}$$
in probability.
\end{thm}
\begin{proof}
Divide $I_n$ in two parts
$$I^1_n=\int_0^{\theta_n}M_n(t)dt \quad \text{ and }\quad I^2_n=\int_{\theta_n}^{\sigma_n}M_n(t)dt.$$
It is easy to prove that ${I^1_n}/{\log n}$ converges to 0 in probability by observing that, almost surely,
\[I^{1}_n\leq M_n({\theta_n})\cdot\theta_n,\]
and using Proposition \ref{summarythetan}.

To study $I^2_n$, we approximate it by the accumulated time for the $M_n({\theta_n})$ seeds to activate, namely
\begin{align*}
\tilde I^{2}_n=\sum_{k=1}^{M_n({\theta_n})} \frac{e_k}{c_2}
\end{align*}
where the $e_k$'s are i.i.d. standard exponential random variables.
The asymptotics of this random variable are easily obtained reproducing the arguments of Section \ref{secgamma}.
First, by Proposition \ref{summarythetan}, we have that $M_n(\theta_n)/\log n\to 2c_1$ in probability.
Second, we use the functional law of large numbers for sums of exponential variables. This leads to the desired result,
$$\lim_{n\rightarrow\infty}\frac{\tilde I^2_n}{\log n}=\frac{2c_1}{c_2}$$
in probability.

Finally, the difference between $I^{2}_n$ and $\tilde I^{2}_n$ can be bounded by $I_{N_n(\theta_n)}+I_{M_n(\theta_n)}$.
Indeed, the variable $I_{N_n(\theta_n)}$ bounds the inactive length resulting from the plants present at time $\theta_n$ and the variable $I_{M_n(\theta_n)}$ bounds the inactive length resulting from the seeds present at time $\theta_n$ that activate and deactivate again.
Its expectation is clearly of order $\log\log n$.
This can be seen repeating the earlier arguments of this proof.
\end{proof}

\section{Sampling formula}
Consider the seed bank coalescent at time $\theta_n$ and go back, through the active part of the genealogical tree, until time zero when there are $n$ active lineages and zero inactive lineages.  During this period of time we observe $n-N_n(\theta_n)$ events divided into two types: branching inside one lineage (corresponding to a coalescence) and appearance of a new lineage (corresponding to a deactivation).
When there are $k$ lineages, the probability that a branching event occurs is
$$\frac{\binom{k+1}{2}}{\binom{k+1}{2}+c_1(k+1)}=\frac{k}{k+2c_1}$$
whereas the probability that a new lineage appears is $\frac{2c_1}{k+2c_1}$.
 This observation leads to make a connection with classical Hoppe's urn and the Chinese restaurant process (with parameter $2c_1$), which are the key tools to prove Ewens' sampling formula for the law of the allele frequency spectrum in the neutral model, see Chapter 1.3 in \cite{Durrett2008}.
 However, in our case, the initial configuration is made of a random number $N_n({\theta_n})$ of tables (old lineages) with one client in each.
By applying results of \cite{W1983}, we can obtain a conditional sampling formula corresponding to observe a certain configuration of lineages that passed through the seed bank and lineages that did not deactivate (until time $\theta_n$).

Now, let $k\leq n$ be a positive integer, we define the sets  

\begin{align*}
A(k,n)=\Bigg\{a_i,b_i\geq 0, i\in[n] : \left. \sum_{i=1}^{n}a_i=k \mbox{ and }
 \sum_{i=1}^{n}i(a_i+b_i)=n \right\rbrace
\end{align*}
and
\begin{align*}
\bar A(k,n)=\Bigg\{a_i\geq 0, i\in[n] : \left. \sum_{i=1}^{n}a_i=k \mbox{ and }
 \sum_{i=1}^{n}ia_i\leq n \right\rbrace.
\end{align*}

From equation (3.3.2) in \cite{W1983}, we obtain the next theorem.
\begin{thm}Let $O_i$ be the number of ``old" blocks of size $i$ (i.e. active blocks of size $i$ at time $\theta_n$)  and let $R_i$ be the number of ``recent" blocks of size $i$ (i.e. inactive blocks of size $i$ at time $\theta_n$). Then

\begin{align}\label{Ewcond}
&\mathbb{P}\left(  O_1=a_1,\dots,O_n=a_n,R_1=b_1,\dots,R_n=b_n\mid N_n({\theta_n})\right)  \nonumber\\
\overset{a.s.}{=}&\dfrac{(n-N_n({\theta_n}))!N_n({\theta_n})!}{(N_n({\theta_n})+2c_1)_{(n-N_n({\theta_n}))}}\prod_{i=1}^n \dfrac{1}{a_i!} \prod_{j=1}^n \dfrac{1}{b_j!}\left(\dfrac{2c_1}{j}\right)^{b_j},
\end{align}

 with $ (a_i,b_i)_{i\in[n]}\in A(N_n(\theta_n),n)$. 
\end{thm}

The notation $x_{(n)}$ stands for the ascending factorial, that is, $x_{(n)}=x(x+1)\dots(x+n-1)$.

{
\begin{rem}
From the latter result and Proposition \ref{summarythetan}, we can obtain an approximate of a sampling formula for large $n$.
\begin{align*}
&\mathbb{P}\left(O_1=a_1,\dots,O_n=a_n,R_1=b_1,\dots,R_n=b_n\right) \\
&=\int_0^{\infty}\P\left( O_1=a_1,\dots,O_n=a_n,R_1=b_1,\dots,R_n=b_n|N_n({\theta_n})=\lfloor z\log n\rfloor\right)\times\\
&\hspace{0.5cm}\P(N_n({\theta_n})=\lfloor z\log n\rfloor)dz\\
&{\sim}\prod_{i=1}^n \dfrac{1}{a_i!} \prod_{j=1}^n \dfrac{1}{b_j!}\left(\dfrac{2c_1}{j}\right)^{b_j}\times\\
&\hspace{0.5cm}\int_0^\infty\frac{\Gamma(n-z\log n+1)\Gamma(z\log n+1)\Gamma(z\log n+2c_1)}{\Gamma(n+2c_1)}.\frac{4c_1c_2}{z^2}e^{-\frac{4c_1c_2}{z}}dz.
\end{align*}
which does not depend on the non-observable variable $N_n(\theta_n)$.
The variables $O_i$ and $R_i$ can be inferred if we are capable of deciding if a present individual has visited the seed bank or not.
This provides a possible method of estimation of the parameters of the seed bank model.
\end{rem}}
From (\ref{Ewcond}), we obtain the probability generating function of the old and recent blocks.

\begin{corollary} Let $O_1,...O_n,R_1,...,R_n$ be random variables with joint density given by (\ref{Ewcond}). Then, their (conditional) probability generating function is

\begin{align}\label{pgf}
\E\left[\prod_{i=1}^n t_i^{O_i} \prod_{j=1}^n s_j^{R_j} |N_n({\theta_n})\right]=\dfrac{(n-N_n({\theta_n}))!N_n({\theta_n})!}{(N_n({\theta_n})+2c_1)_{(n-N_n({\theta}_n))}}\times\nonumber\\
\sum_{\substack{a_1,...,a_n,b_1,...,b_n\in A(N_n(\theta_n),n)}}\prod_{i=1}^n\dfrac{(t_i)^{a_i}}{a_i!}\prod_{j=1}^n\dfrac{1}{b_j!}\left(\frac{2c_1s_j}{j}\right)^{b_j}.
\end{align}

\end{corollary}

%If the conditions in set (\ref{set})were not met, we could rewrite (\ref{pgf}) as,
%\begin{align}\label{pgf2}
%\dfrac{(n-N_n({\theta_n}))!N_n({\theta_n})!}{(N_n({\theta_n})+2c)_{(n-N_n({\theta}_n))}}
% \exp\left\lbrace \sum_{i=0}^n\left(  t_i+ \frac{2cs_i}{i}\right) \right\rbrace. 
% \end{align}

Following the idea of Watterson \cite{W1983}, we use two artificial variables, $u\in(-1,1)$ and $v\in(-1,1)$. They will help us to rewrite (\ref{pgf}) in a simpler way. First, observe that for $(a_i,b_i)\in A(k,n)$, 
\begin{equation*}
\prod_{i=1}^n (uv^i)^{a_i} \prod_{j=1}^n (v^j)^{b_j}=u^{\sum_{i=1}^n a_i} v^{\sum_{i=1}^n i(a_i+b_i)}=u^kv^n .
\end{equation*} 
Now, let $c_{k,n}$ be the multiplying coefficient of $u^kv^n$ in $\exp\left\lbrace \sum_{i=1}^n  uv^it_i+\sum_{j=1}^\infty \frac{2c_1}{j}s_jv^j\right\rbrace $. We can rewrite (\ref{pgf}) as

\begin{align}\label{pgf3}
\E\left[\prod_{i=1}^n t_i^{O_i} \prod_{j=1}^n s_j^{R_j} |N_n({\theta_n})\right]=\dfrac{(n-N_n({\theta_n}))!N_n({\theta_n})!}{(N_n({\theta_n})+2c_1)_{(n-N_n({\theta}_n))}}c_{N_n({\theta}_n),n}.
\end{align}
%The factorization in (\ref{Ewcond}) shows that,  if it were not for the restrictions (\ref{set}) , $O_1,...,O_n$ would behave as independent Poisson variables with mean 1, and $R_1,...,R_n$ with means $\dfrac{2c}{1},\dfrac{2c}{2},...,\dfrac{2c}{n}$, respectively. 

From  this relation, we obtain the probability generating function of the lineages that have not gone through the seed bank at time $\theta_n$.

\begin{corollary}\label{copgf} Let $O_i$ be the number of ``old" blocks of size $i$ (i.e. active blocks of size $i$ at time $\theta_n$). Then, the joint probability generating function of $O_1,O_2,...,O_n$ is
\begin{align}\label{pgfa}
\E\left[\prod_{i=1}^n t_i^{O_i}  |N_n({\theta_n})\right]=\sum_{a_1,..,a_n\in\bar{A}(N_n({\theta_n}),n)}\dfrac{N_n({\theta_n})!}{a_1!a_2!...a_n!}t_1^{a_1}t_2^{a_2}\cdot\cdots t_n^{a_n}\dfrac{\binom{2c_1+n-z-1}{n-z}}{\binom{2c_1+n-1}{n-N_n{\theta_n}}}
\end{align}

where $z={\sum_{i=1}^nia_i}.$ 
\end{corollary}

\begin{proof}
 First, we will write explicitly the term $c_{k,n}$ when $s_j=1$ for all $j$. Observe that,   
\begin{align*}
 \exp\left\lbrace \sum_{i=1}^n  uv^it_i+\sum_{j=1}^\infty\frac{2c_1}{j}v^j\right\rbrace 
 &=(1-v)^{-2c_1} \exp\left\lbrace u\sum_{i=1}^{n}  v^it_i\right\rbrace\\
 &=(1-v)^{-2c_1}\sum_{k=0}^{\infty}\dfrac{\left[ u\sum_{i=1}^{n}  v^it_i\right] ^{k}}{k!}.
\end{align*}

It implies that the coefficient of $u^k$ in the latter expression is
\begin{align*}
\dfrac{\left[ \sum_{i=1}^{n}  (v^it_i)\right] ^{k}}{k!}(1-v)^{-2c_1}=&\dfrac{\left[ \sum_{i=1}^{n}  (v^it_i)\right] ^{k}}{k!}\left( \sum_{j=0}^{\infty} \binom{2c_1+j-1}{j}v^j\right).
\end{align*}
Now, we need to find the coefficient of $v^n$ in the latter expression.
First, observe that
\begin{align*}
\left[ \sum_{i=1}^{n}  (v^it_i)\right] ^{k}=\sum_{{a_1+...+a_n=k}} \dfrac{k!}{a_1!a_2!...a_n!}t_1^{a_1}t_2^{a_2}\cdot\cdots t_n^{a_n}v^{z}
\end{align*}
where $z={\sum_{i=1}^nia_i}$.
Hence , for $z\leq n$, the coefficient of $v^{n-z}$ in the expression $\left( \sum_{j=0}^{\infty} \binom{2c_1+j-1}{j}v^j\right)$ is $\binom{2c_1+n-z-1}{n-z}$.
 So,

\begin{align*}
 c_{k,n} =\dfrac{1}{k!}\sum_{a_1,...,a_n\in\bar A(k,n)} \dfrac{k!}{a_1!a_2!...a_n!}t_1^{a_1}t_2^{a_2}\cdot\cdots t_n^{a_n}\binom{2c_1+n-z-1}{n-z}.
\end{align*}

Thus, replacing $c_{N_n({\theta}_n),n}$ and $s_j=1$ for all $j$ in (\ref{pgf3}) we have the result.
\end{proof}

From the previous corollary we obtain the joint distribution of the lineages which have not gone through the seed bank at time $\theta_n$.
\begin{align*}
\P\left[ O_1=a_1,....,O_n=a_n|N_n({\theta_n})\right]&\overset{a.s.}{=}\frac{N_n({\theta_n})!}{a_1!a_2!\cdots a_n!}\frac{\binom{2c_1+n-z-1}{n-z}}{\binom{2c_1+n-1}{n-N_n({\theta_n})}}
\end{align*}
when $a_1,\dots,a_n\in\bar A(N_n(\theta_n),n)$.

Now, by taking $t_i=t^i$ and $s_j=1$ for all $i,j\in[n]$ in (\ref{pgf3}), and finding the corresponding coefficient $c_{N_n({\theta}_n),n}$, we obtain the conditional probability generating function of the number of lineages at time zero that has not been through the seed bank until time $\theta_n$
\begin{align*}
\E\left[  t^{\sum_{i=1}^niO_i} |N_n({\theta_n})\right]&=\sum_{z=N_n({\theta_n}) }^nt^z \dfrac{\binom{2c_1+n-z-1}{n-z}\binom{z-1}{z-N_n({\theta_n})}}{\binom{2c_1+n-1}{n-N_n({\theta_n})}}.
\end{align*}

%Also, we obtain the probability generating function of the number of active lineages at time $\theta_n$ 
%\begin{align}\label{pgfactive}
%\E\left[  t^{\sum_{i=1}^nO_i} |N_n({\theta_n})\right]&=t^{N_n({\theta_n})}.
%\end{align}
%It is obtained from (\ref{pgf3}) with $t_i=t$ and $s_j=1$ for all  $i,j\in[n]$, and developing the expression $\exp\left\lbrace \sum_{i=0}^{\infty} \left(  uv^it+ \frac{2c}{i}v^i\right)\right\rbrace$ to find the coefficient $c_{N_n({\theta_n}),n}$.

% First, we will to find $c_{k,n}$, so
%\begin{align}
 %\exp\left\lbrace \sum_{i=0}^{\infty} \left(  uv^it+ \frac{2c}{i}v^i\right)\right\rbrace&= \exp\left\lbrace ut\sum_{i=0}^{\infty} \left(  v^i\right) - 2c\log(1-v)\right\rbrace\\
 %=&\exp\left\lbrace \dfrac{utv}{(1-v)} - 2c\log(1-v)\right\rbrace\\
% =&\sum_{k=0}^{\infty}\dfrac{\left[ utv(1-v)^{-1}\right] ^{k}}{k!}(1-v)^{-2c}.
%\end{align}
%So, the coefficient for $u^k$ is
%\begin{align*}
%\dfrac{\left[ tv(1-v)^{-1}\right] ^{k}}{k!}(1-v)^{-2c}=&\dfrac{t^kv^k(1-v)^{-(2c+k)}}{k!}=\dfrac{t^k}{k!}%\left( \sum_{j=0}^{\infty} \binom{2c+k+j-1}{j}v^{k+j}.\right), 
%\end{align*}
%then, the coefficient for $u^kv^{n}$, is
%\begin{align*}
%\dfrac{t^k}{k!} \binom{2c+n-1}{n-k}.
%\end{align*}

%Thus, replacing $c_{N_n(\theta_n),n}$ and $s_j=1$ for all $j\in[n]$ in (\ref{pgf3}) we have the result.
%\textcolor{red} {lo ultimo lo podriamos quitar del paper, sòlo lo puse que ustedes vean como salen las cuentas.} 

Finally, from (\ref{pgf3}), by taking $t_i=1$ for all $i\in[n]$, and from (\ref{pgfa}) we can find the conditional expectations of $O_j$ and $R_j$ for all ${j=1,2,...n-N_n({\theta_n})}$,
$$
\E\left( O_j |N_n({\theta_n})\right)
=N_n({\theta_n})\dfrac{\binom{2c_1+n-j-1}{n-j-N_n({\theta_n})+1}}{\binom{2c_1+n-1}{n-N_n({\theta_n})}}
$$
and
$$\E\left( R_j |N_n({\theta_n})\right)=\frac{2c_1}{j}\dfrac{\binom{2c_1+n-j-1}{n-j-N_n({\theta_n})}}{\binom{2c_1+n-1}{n-N_n({\theta_n})}}.
$$

%Last equations and  Propositions (\ref{cotasupN}), (\ref{cotainfN})  suggest an approximation for the expectation of the number of old blocks% we know that $N_n(\theta_n)$ is of order $\log n$  so, we can obtain an approximation of the expectation of $O_j$ and $R_j$ for all $j\in[n]$%$\P((\log n)^{1-\varepsilon}<N_n(\theta_n)<(\log n)^{1+\varepsilon}$ 
%
%$$
%\dfrac{(n-j-\log n+2)_{(j-1)}(2c+\log n-1)}{(2c+n-j)_{(j)}},\mbox{ } j\in[n]
%$$
%and for the recent blocks
%$$\frac{2c}{j}\dfrac{(n-j-\log n+1)_{(j)}}{(2c+n-j)_{(j)}},\mbox{ } j\in[n].
%$$

%Finally, we calculate the number of blocks in the kingman coalescent at time $\dfrac{1}{\log n}$ to compare this with the seed-bank coalescent. 
%\begin{align*}
%\sum_x^n\dfrac{2}{i(i-1)}=2\left( \dfrac{1}{x-1}-\dfrac{1}{n}\right) =\frac{1}{\log n}
%\end{align*}
%\begin{align*}
%x-1=\dfrac{2\log n}{1+\dfrac{2\log n}{n}}
%\end{align*}

{\bf Acknowledgement.}
This project was partially supported by CoNaCyT grant FC-2016-1946.

\end{document}